\documentclass[12pt]{article}
\usepackage{hyperref}
\usepackage{cmap}                        
\usepackage[cp1251]{inputenc}            
\usepackage[english]{babel}
\usepackage[margin=2.2cm]{geometry}
\usepackage{amssymb,amsmath, amsthm, amscd,ifthen}
\usepackage{graphicx}

\newtheorem*{thm*}{Theorem}
\newcommand{\ff}{{\mathcal F}}

\newtheorem*{cla*}{Claim}

\newtheorem{thm}{Theorem}

\newtheorem{cla}[thm]{Claim}

\date{}

\DeclareMathOperator{\E}{\mathrm E}

\title{Families of sets with no matchings of sizes 3 and 4}
\author{Peter Frankl, Andrey Kupavskii\footnote{Moscow Institute of Physics and Technology, Ecole Polytechnique F\'ed\'erale de Lausanne; Email: {\tt kupavskii@yandex.ru} \ \ Research supported by the grant RNF~16-11-10014.}}

\date{}

\begin{document}
\maketitle
\begin{abstract} In this paper, we study the following classical question of extremal set theory: what is the maximum size of a family of subsets of $[n]$ such that no $s$ sets from the family are pairwise disjoint? This problem was first posed by Erd\H os and resolved for $n\equiv 0, -1\ (\mathrm{mod }\ s)$ by Kleitman in the 60s. Very little progress was made on the problem until recently. The only result was a very lengthy resolution of the case $s=3,\ n\equiv 1\ (\mathrm{mod }\ 3)$ by Quinn, which was written in his PhD thesis and never published in a refereed journal. In this paper, we give another, much shorter proof of Quinn's result, as well as resolve the case $s=4,\ n\equiv 2\ (\mathrm{mod }\ 4)$. This complements the results in our recent paper, where, in particular, we answered the question in the case $n\equiv -2\ (\mathrm{mod }\ s)$ for $s\ge 5$.

\end{abstract}
\section{Introduction}
Let $[n] := \{1,2,\ldots, n\}$ be the standard $n$-element set and $2^{[n]}$ its power set. A subset $\mathcal F\subset 2^{[n]}$ is called a \textit{family}. For $0\le k\le n$, let ${[n]\choose k}$ denote the family of all $k$-subsets of $[n]$.

For a family $\ff$, let $\nu(\ff)$ denote the maximum number of pairwise disjoint members of $\ff$. Note that $\nu(\ff)\le n$ holds unless $\emptyset \in \ff$. The fundamental parameter $\nu(\ff)$ is called the \textit{independence number} or \textit{matching number}.

Denote the size of the largest family  $\ff\subset 2^{[n]}$ with $\nu(\ff)<s$ by $e(n,s)$.
The following  classical result was obtained by Kleitman.
\vskip+0.3cm
\noindent{\bf Kleitman's Theorem} (\cite{Kl}) Let $s\ge 2, m\ge 1$ be integers.  Then the following  holds.
\begin{align}\label{eq001} \text{For }n=sm-1,\ \text{ we have } \ \ \ e(n,s) &= \sum_{m\le t\le n}{n\choose t},\\
\label{eq002} \text{for }n=sm,\ \text{\ \ \ \ \ \,  we have }\ \ \ e(n,s)&= \frac{s-1}s{n\choose m}+\sum_{m+1\le t\le n}{n\choose t}.
\end{align}
\vskip+0.2cm
The value $e(ms-1,s)$ is attained on the family of all sets of size greater than or equal to $m$. The following matching example for  (\ref{eq002}) was proposed by Kleitman:
$$\big\{K\subset [sm]:|K|\ge m+1\big\}\cup {[sm-1]\choose m}.$$
(Note that ${sm-1\choose m} = \frac {s-1}s{sm\choose m}$.) Let us mention that for $s=2$ both bounds (\ref{eq001}) and (\ref{eq002}) reduce to $e(n,2)= 2^{n-1}$. This easy statement was proved already by Erd\H os, Ko and Rado \cite{EKR}.

Although (\ref{eq001}) and (\ref{eq002}) are beautiful results, for $s\ge 3$ they leave open the cases of $n\not\equiv 0,-1 (\mathrm{mod}\ s).$ For $s=3$, the only remaining case was solved by Quinn \cite{Q}. However, his argument is very lengthy and was never published in a refereed journal. In this paper, we reprove his result, as well as extend it to the case $n=4m+2, s=4$.
\begin{thm}\label{thm1}Fix an integer $m\ge 1$. Then for $s=3,4$ and $n=sm+s-2$ we have
\begin{equation}\label{eq003}e(n,s)= {n-1\choose m-1}+\sum_{m+1\le t\le n}{n\choose t}.\end{equation}\end{thm}
The following $s$-matching-free family shows that ``$\ge$'' holds in the equality above for any $s\ge 3$ and $n=sm+s-2$.
$$\big\{L\subset [n]: |L|\ge m+1\big\}\cup\Big\{L\in {[n]\choose m}: 1\in L\Big\}.$$

Theorem~\ref{thm1} bridges the gap that was left between Quinn's result and the result of the paper \cite{FK7}, where we verified the same statement for $n = sm+s-2, s\ge 5$. Contrary to the intuition, the problem gets easier as $s$ becomes larger, and thus the proof for $s=3,4$ is more intricate than that of \cite{FK7}.

The proof is based on a non-trivial averaging technique somewhat in the spirit of Katona's circle method \cite{K}: we choose a certain configuration of sets, show that the intersection of a family satisfying the conditions of Theorem \ref{thm1} with {\it each} such configuration cannot be too large and then average over all such configurations. However, the configuration is quite complicated, the sets in the configuration actually have weights, and, in order to bound the weighted intersection of the family with each configuration,  we use some kind of discharging method.

The method we develop here has proved to be very useful and was already used in several papers. In a recent paper \cite{FK13}, we applied it to completely resolve the following problem studied by Kleitman: what is the maximum cardinality of a family $\ff\subset 2^{[n]}$ that does not contain two disjoint sets $F_1, F_2$, along with their union $F_1\cup F_2$? We refer the reader to the papers \cite{FK7}, \cite{FK13} for a more detailed introduction to the topic and, in particular, to \cite{FK7} for the discussion of the case of general $n,s$. See also \cite{FK14}, where the method we developed was applied.



We note that \eqref{eq001} and \eqref{eq002}, along with more general statements, are proved using a simpler version of our technique in \cite{FK8}.

\section{Preliminaries}
Recall that $\mathcal F$ is called an \textit{up-set} if for any $F\in \mathcal F$ all sets that contain $F$ are also in $\mathcal F$. Since we aim to upper bound the sizes of families $\ff$ with $\nu(\ff)<s$, we may restrict our attention to the families that are up-sets, which we assume for the rest of the paper.\\

We are going to use the following inequality in the proofs:
\begin{equation}\label{eq005}\sum_{j=1}^{k-1}{sm+s-2\choose j}\le \frac 1{s-2}{sm+s-2\choose k}\ \ \ \ \text{for any }k\le m,\ s\ge 3.\end{equation}
Indeed, we have $\frac{{sm+s-2\choose k-j}}{{sm+s-2\choose k-j-1}} = \frac{sm+s-2-k+j+1}{k-j}\ge s-1$ for any $j\ge 0$ and $k\le m$, so by the formula for the summation of a geometric progression, $$\sum_{j=1}^{k-1}{sm+s-2\choose j}\le \frac{\frac 1{s-1}}{1-\frac 1{s-1}}{sm+s-2\choose k} = \frac 1{s-2}{sm+s-2\choose k}.$$
\vskip+0.2cm

\section{Proof of Theorem \ref{thm1} for $s=3$}


We first prove the theorem for $m\ge 3$.
Suppose that $m\ge 3$ and put $n:=3m+1$ for this section. Consider a family $\ff\subset 2^{[n]}$ with $\nu(\ff)<3$. Take an arbitrary cyclic permutation $\sigma$ (assumed in what follows to be the identity permutation for simplicity) and fix  three disjoint $m$-element sets that form arcs in that permutation. This is what we call a \textit{triple}. For $x\in[3m+1]$, the {\it $x$-triple} is the triple of $m$-sets that do not contain the element $x$. It is clear that there is a one-to-one correspondence between the $x$'s and the triples. For each triple, we define three groups of sets of sizes $1,\ldots, m+3$  and assign them weights.  We call this ensemble of sets an {\it $x$-family}. Note that the arithmetic operations in the definitions of the sets are performed modulo $n$.


We define three groups of sets, indexed by $i=0,1,2$. In what follows, we define group $i$. The $i$-th $m$-set $H_i^{(m)}(x)$ in the $x$-family has the form $\{im+x+1,\ldots, im+x+m\}$. The set $H_i^{(m-j)}(x)$ of size $m-j$ for $j<m$ has the form $\{im+j+x+1,\ldots, im+x+m\}$. That is, it consists of the last $m-j$ elements of the $m$-set, if seen in the clockwise order. The sets $\emptyset\subset H_i^{(1)}(x)\subset H_i^{(2)}(x)\subset\cdots\subset H_i^{(m)}(x)$ form a full chain. The definition of the sets of size $\ge m+1$ is less straightforward. Each of the $(\ge m+1)$-sets in the $i$-th group contains the corresponding $m$-set. The $(m+1)$-set
$$H_i^{(m+1)}(x;x):=H_i^{(m)}(x)\cup\{x\}$$ in group $i$ is called \textit{central}. Note that the extra element it has is the element that was left out by the $m$-sets and so $H_i^{(m+1)}(x;x)$ is disjoint of $H_j^{(m)}(x)$ for  $j\ne i$. The two others $$H_i^{(m+1)}(j;x):=H_i^{(m)}(x)\cup\{mj+x+1\}\ \ \ \text{for }j\in\{0,1,2\}- \{i\}$$
are called \textit{lateral} and are disjoint of the corresponding $H_j^{(m-1)}(x)$ and the remaining $m$-set $H_{j_1}^{(m)}(x)$, where $\{j_1,j,i\}=\{0,1,2\}$. For each $j\in\{0,1,2\}- \{i\}$, we define two $(m+2)$-element sets: the \textit{central} set $$H_i^{(m+2)}(x,j;x):=H_i^{(m)}(x)\cup\{mj+x+1,x\}$$ and the \textit{lateral} set $$H_i^{(m+2)}(j,j;x):=H_i^{(m)}(x)\cup\{mj+x+2,mj+x+3\}.$$ The former ones are disjoint of $H^{(m-1)}_j(x)$ and the $m$-set from the remaining $j_1$-th group, where $\{j_1,j,i\}=\{0,1,2\}$, while the latter ones are disjoint of $H^{(m-3)}_j(x)$. Note that $H_{j}^{(m+2)}(x,i;x)$ and $H_{j_1}^{(m+2)}(i,i;x)$ are disjoint for $\{j_1,j, i\} = \{0,1,2\}$.
Finally, we have one $(m+3)$-set in each group: $$H_i^{(m+3)}(x):=H_i^{(m)}(x)\cup\{x,x+1,m+x+1,2m+x+1\}.$$ It is disjoint of the $(m-1)$-sets from the other groups.


Each set in each group gets a weight. We denote by $w_{k}$ the weight of the $k$-element sets, with possible superscripts $l,c$ depending on whether the set is lateral or central, respectively. Put \begin{equation}\label{eq04}\alpha := \frac{(3m+2)m}{4(2m+3)(2m+1)}, \ \ \ \ \alpha':=1-2\alpha.
\end{equation}
Note that $\alpha'>\alpha$. The weights are as follows ($j\ge 0$):
\begin{align}\label{eq05}w_{m-j}&:={n\choose m-j};  \ \ \ \ \ \ w^l_{m+1}:=\alpha{n\choose m+1};\ \ \ \ \ \ w^c_{m+1}:=\alpha'{n\choose m+1};\notag\\
w_{m+3} &:= {n\choose m+3}; \ \ \ \ \ \ w^l_{m+2}:=\frac 1{8}{n\choose m+2};\ \ \ \ \ \   w^c_{m+2}:=\frac{3}{8}{n\choose m+2}.
\end{align}

\begin{figure}[t] \begin{center}  \includegraphics[width=38mm]{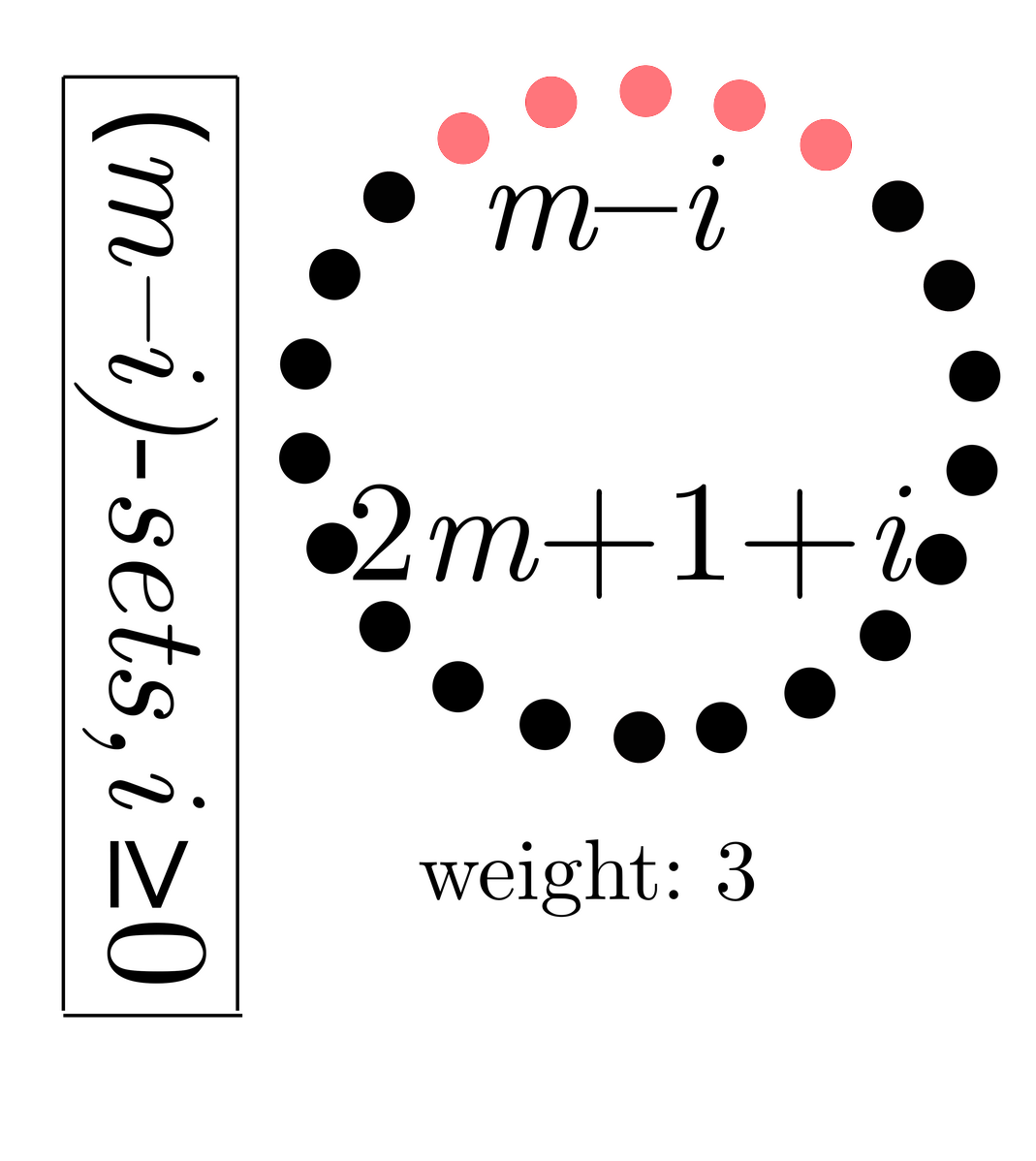}  \includegraphics[width=125mm]{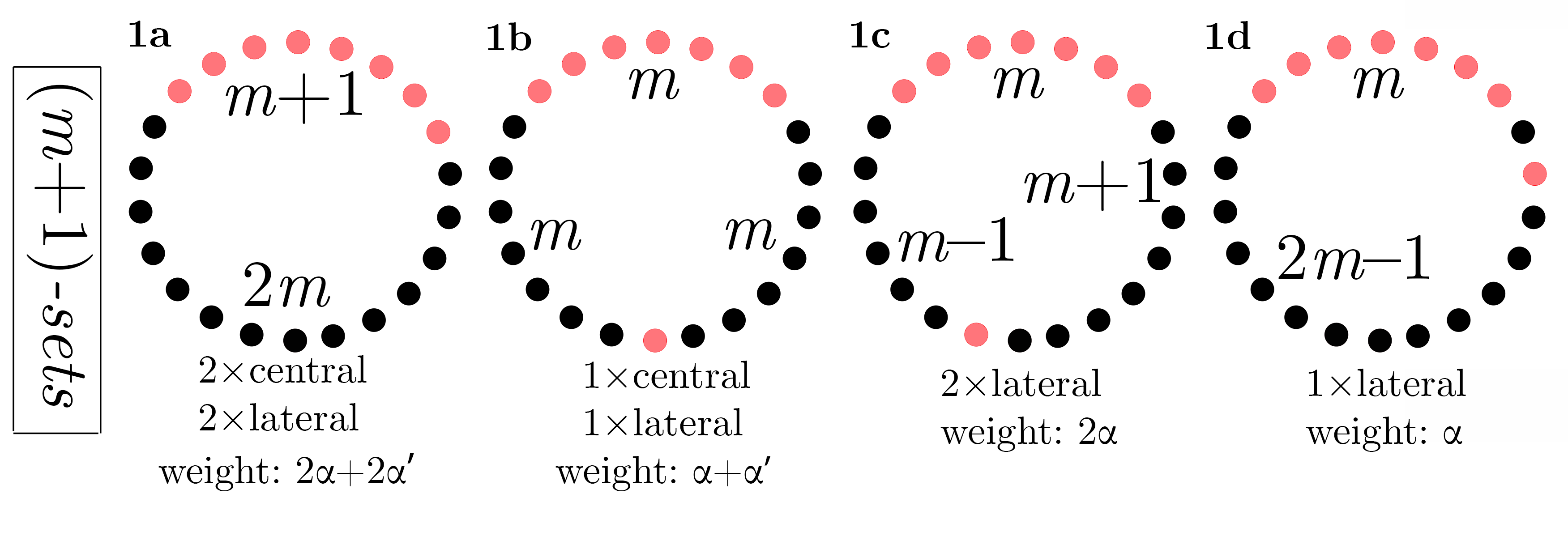}\\ \includegraphics[width=125mm]{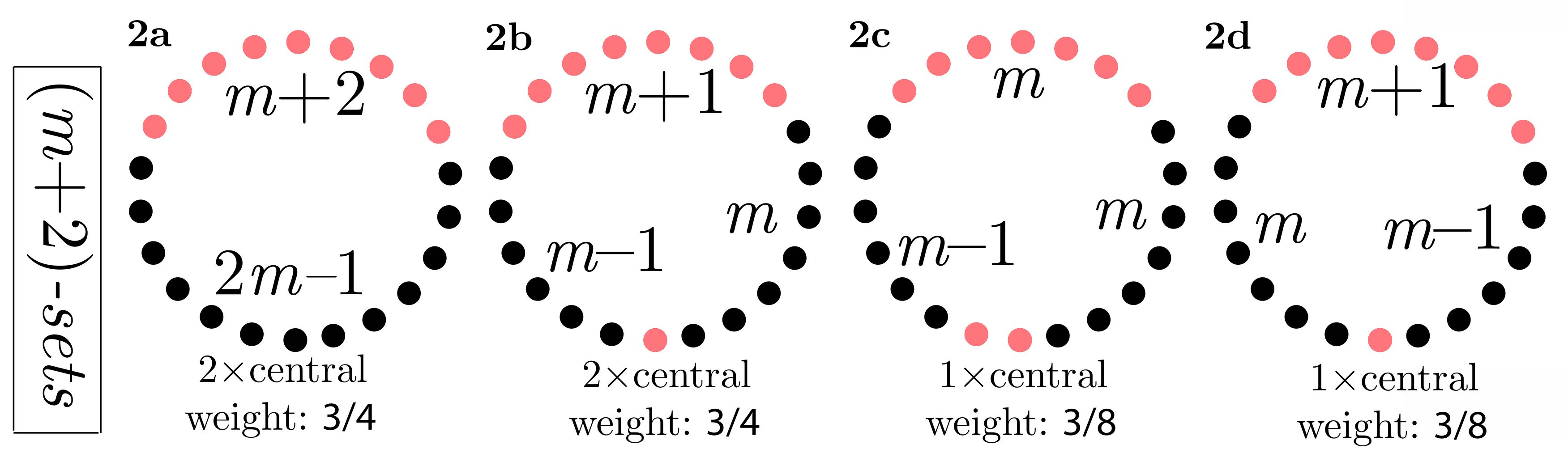}\includegraphics[width=38mm]{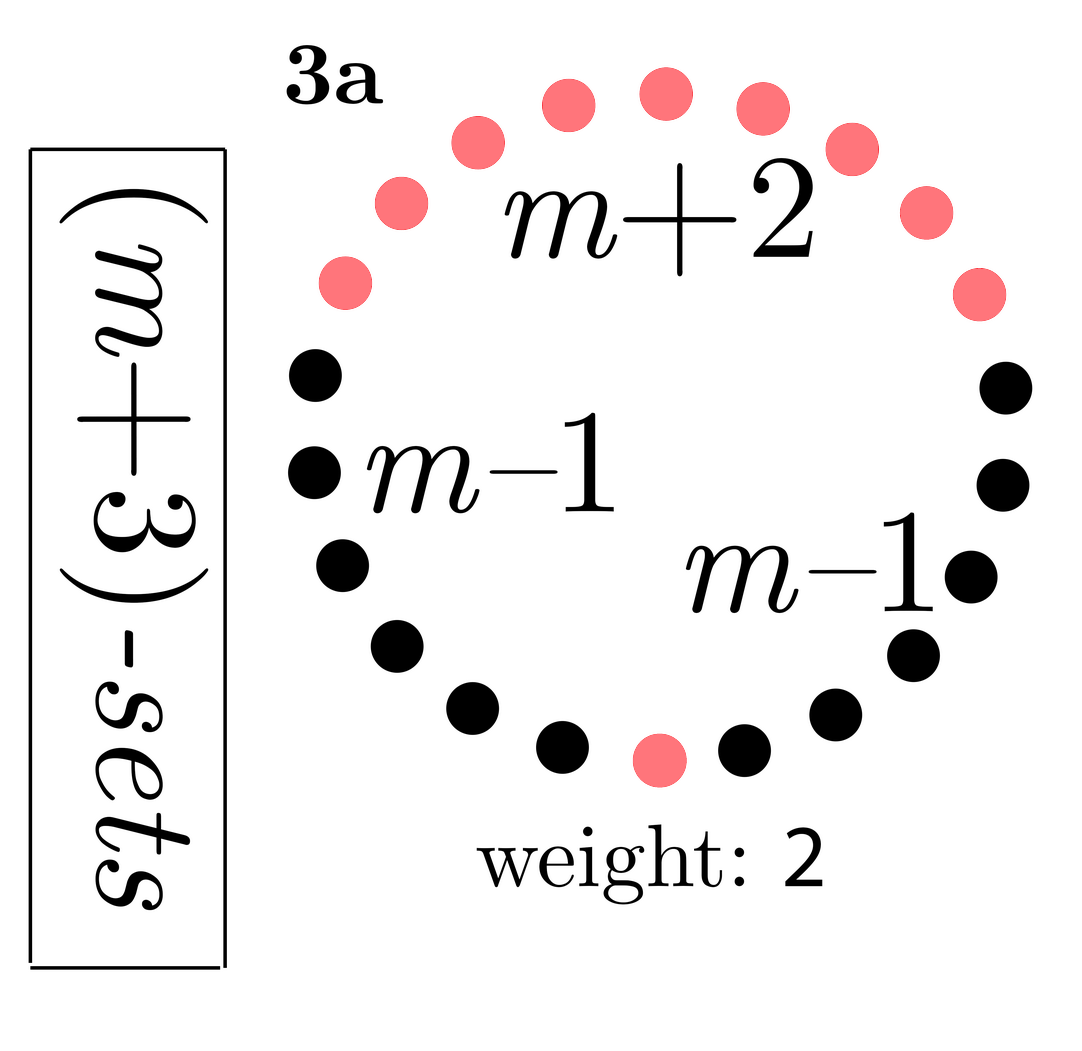}\\ \includegraphics[width=125mm]{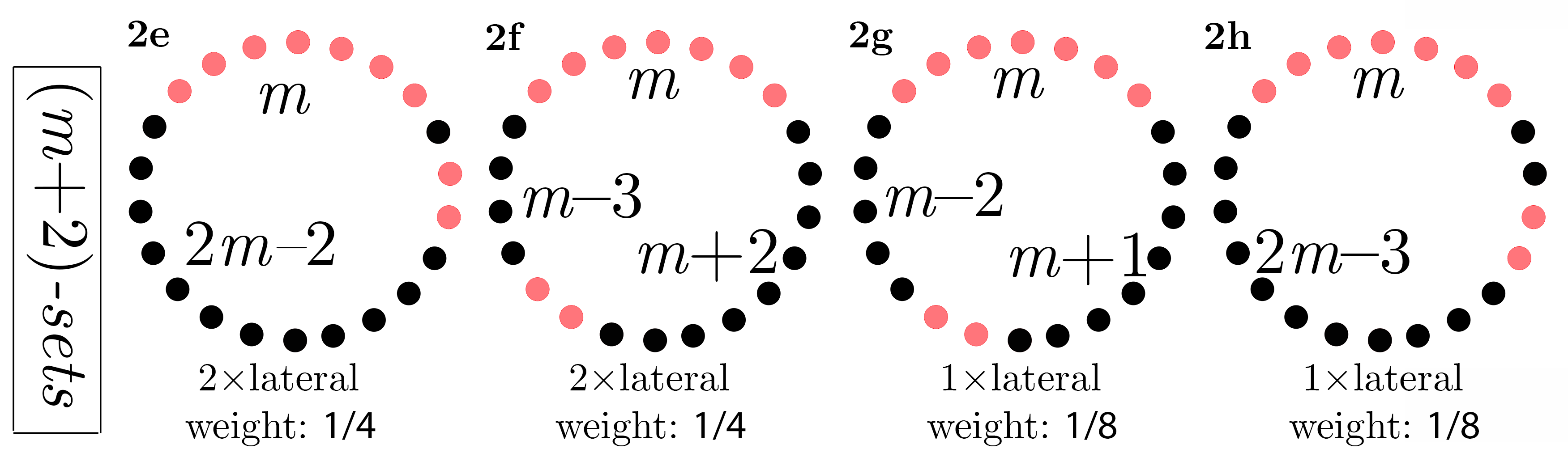}\includegraphics[width=38mm]{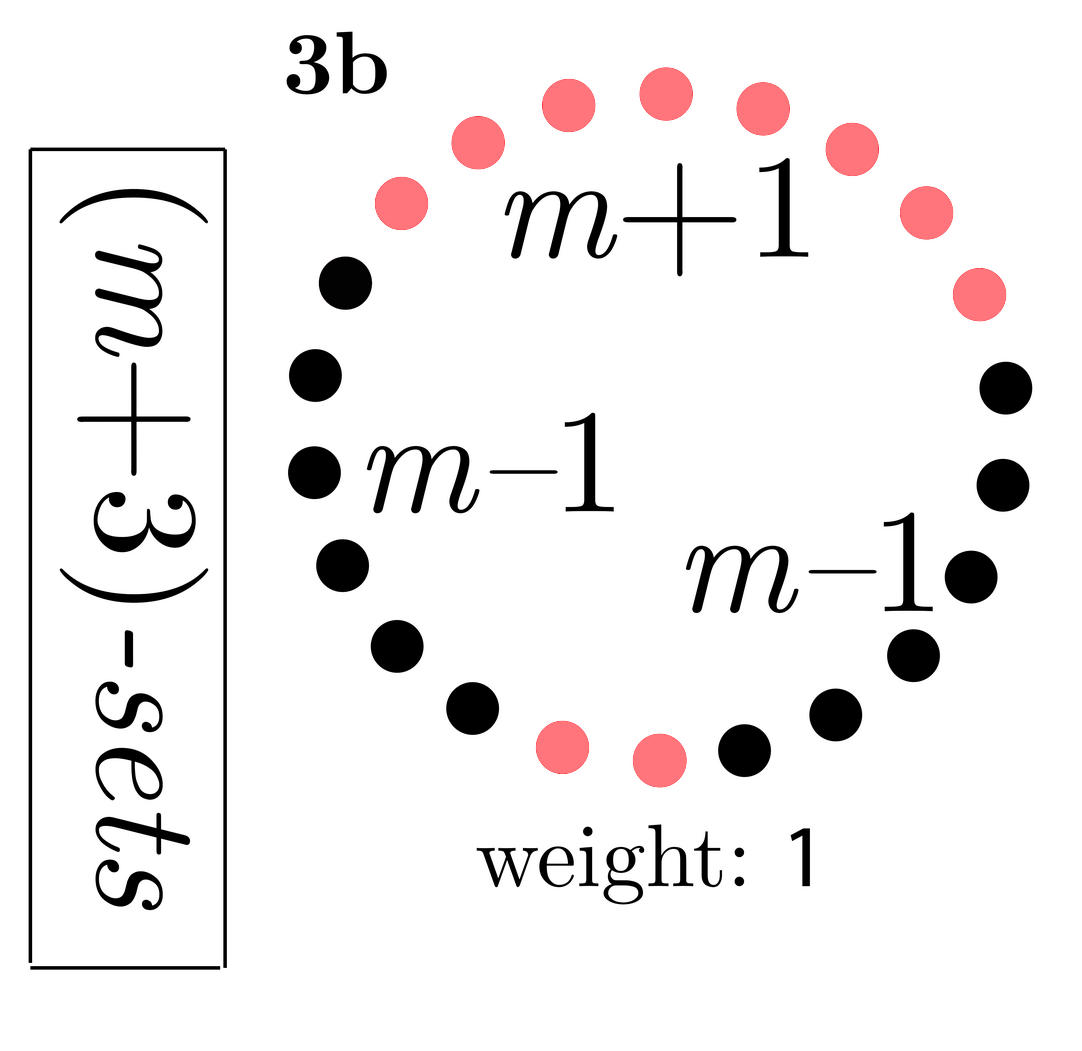}\\ \caption{All types of sets in $\mathcal G(\sigma)$ modulo rotation. The elements of the ground set are represented by dots, and the red dots represent the elements that are contained in the set. Inside the circles we specify the length of any interval of points that simultaneously either belong or do not belong to the set (excluding the intervals of length 1 and 2). The weights of $k$-sets are specified and divided by ${n\choose k}$ for shorthand. For $(m+1)$- and $(m+2)$-sets we also mention, how many times does a given set appear as a central and lateral   set.}\label{fig1}\end{center} \end{figure}

Thus, the total weight of all sets in an $x$-family is $3\sum_{k=1}^{m+3}{n\choose k}$. Each set $F$ that appears in several $x$-families accumulates all the weight $w(F)$ that it was assigned. On Fig.~\ref{fig1}, we listed all the types of sets that are assigned non-zero weights, together with the corresponding weights. We recommend the reader to verify Fig.~\ref{fig1}, since we shall use the information provided on the figure later in the proof! The elements of the ground set are placed on the circle and the sets are represented modulo rotation. We denote by $\mathcal G(\sigma)$ the family of all sets that got  non-zero weight for a given permutation $\sigma$. Note that, for each $j=1,\ldots, m+3$, we have $$\sum_{G\in\mathcal G(\sigma)\cap{[n]\choose j}}w(G) = 3n{n\choose j}.$$

\begin{cla}\label{cla2} To prove the theorem for $s=3$, it is sufficient to show that for any $\sigma$ we have
\begin{equation}\label{eq76}\sum_{F\in \mathcal F\cap \mathcal G(\sigma)}w(F)\le 3n\Big({n-1\choose m-1}+\sum_{k=m+1}^{m+3}{n\choose k}\Big) = 3m{n\choose m}+3n\sum_{k=m+1}^{m+3}{n\choose k}.\end{equation}
\end{cla}
\begin{proof}
For an event $A$, denote by $I[A]$ its indicator random variable. Denote $id$ the identity permutation. Indeed, if we take a permutation $\sigma$ uniformly at random, then, for each $j=1,\ldots, m+3$, we have $$\E_{\sigma}\bigg[\sum_{F\in \mathcal F\cap \mathcal G(\sigma)\cap {[n]\choose j}}w(F)\bigg] = \sum_{F\in \mathcal F\cap {[n]\choose j}}\E_{\sigma}\bigg[\sum_{G\in \mathcal G(id)\cap {[n]\choose j}}I[\sigma(F)=G]w(G)\bigg]=$$
$$=\sum_{F\in \mathcal F\cap {[n]\choose j}}\sum_{G\in \mathcal G(id)\cap {[n]\choose j}}\Pr[\sigma(F)=G]w(G)=\sum_{F\in \mathcal F\cap {[n]\choose j}}\sum_{G\in \mathcal G(id)\cap {[n]\choose j}}\frac {w(G)}{{n\choose j}} = $$ $$=\sum_{F\in \mathcal F\cap {[n]\choose j}} 3n = 3n\Big|\ff\cap{[n]\choose j}\Big|.$$

Therefore, \eqref{eq76} implies that
$$\sum_{j=1}^{m+3}\Big|\ff\cap{[n]\choose j}\Big| = \frac 1{3n} \E_{\sigma}\bigg[\sum_{F\in \mathcal F\cap \mathcal G(\sigma)}w(F)\bigg]\le {n-1\choose m-1}+\sum_{k=m+1}^{m+3}{n\choose k},$$
which implies the statement of the theorem for $s=3$.\end{proof}

Our strategy to prove (\ref{eq76}) is as follows. For a set $F\in \mathcal G(\sigma)$, we define the charge $c(F)$ to be equal to $w(F)$ if $F\in\mathcal F$, and to be $0$ otherwise. Clearly, $\sum_{F\in \mathcal G(\sigma)}c(F) = \sum_{F\in \mathcal F\cap \mathcal G(\sigma)}w(F)$. If among the $(\le m-1)$-sets in $\mathcal G(\sigma)$ there are no sets from $\mathcal F$, as well as there are at most $m$ $m$-sets, then we are done since each $m$-set appears in exactly three $x$-families. Otherwise, certain $(\ge m+1)$-sets do not appear in $\mathcal F$. Then we transfer (a part of) the charge of the $(\le m)$-sets to the $(\ge m+1)$-sets that have zero charge. We show that the total charge transferred to each $(\ge m+1)$-set is at most its weight. As a result of this procedure, the $(\le m-1)$-sets will have zero total charge, the $m$-sets will have total charge at most $3m{n\choose m}$, and each $(\ge m+1)$-set will have a charge not greater than its own weight. This will obviously conclude the proof of the theorem.

Next, we design a charging scheme that satisfies the above requirements. For the sets of size at most $m-1$, we transfer their charge within each $x$-family, assuring that the charge that we transferred to a bigger set in one $x$-family is smaller than the weight that this bigger set got \textit{from this $x$-family}. See Table~1 for all the triples of pairwise disjoint sets we use in the proof. The reader is welcome to verify that all the triples are actually disjoint. \\
\begin{table}\begin{center}
\begin{tabular}{|c|c|c|}
  \hline
  $H_{j_1}^{(m-3)}(x)$ & $H_{j_2}^{(m+2)}(j_1,j_1;x)$ & $ H_{j_3}^{(m+2)}(x,j_1;x)$ \\ \hline
  $H_{j_1}^{(m-1)}(x)$ & $H_{j_2}^{(m-1)}(x)$ & $H_{j_3}^{(m+3)}(x)$ \\
  \hline
  $H_{j_1}^{(m-1)}(x)$ & $H_{j_2}^{(m)}(x)$ & $ H_{j_3}^{(m+2)}(x,j_1;x)$ \\
  \hline
  $H_{j_1}^{(m-1)}(x)$ & $H_{j_2}^{(m+1)}(x;x)$ & $ H_{j_3}^{(m+1)}(j_1;x)$\\
  \hline
  $H_{j_1}^{(m)}(x)$ & $H_{j_2}^{(m)}(x)$ & $ H_{j_3}^{(m+1)}(x;x)$\\
  \hline
\end{tabular}
\end{center}
\caption{The list of all types of triples of pairwise disjoint sets that we employ in the proof. We assume that $\{j_1,j_2,j_3\}=\{0,1,2\}$. The triples are listed in the order they appear in the proof.}
\end{table}

\textbf{Stage~1. Transferring charge from the $\mathbf{(\le m-3)}$-sets to $\mathbf{(m+2)}$-sets}.\\ Assume that, for some $x\in [n]$  and $i\in\{0,1,2\}$, the set $H^{(m-3)}_i(x)$ is in the family.  Choose $j_1,j_2$ such that $\{j_1,j_2,i\} =\{0,1,2\}$. Then at least one set from each of the two pairs
$\bigl(H_{j_1}^{(m+2)}(i,i;x),$ $ H_{j_2}^{(m+2)}(x,i;x)\bigr),$ $ \bigl(H_{j_2}^{(m+2)}(i,i;x),$ $ H_{j_1}^{(m+2)}(x,i;x)\bigr)$ is missing from $\mathcal F$. We transfer $\frac 12$ of the charge of the subsets $H_i^{(k)}(x)$, $k\le m-3$, which is at most $\frac 12 \sum_{k=1}^{m-3}w_{k}$,  to some two of these missing sets. We again remark that, for each $x$, we transfer only the part of the charge of the sets $H_i^{(k)}(x)$ that they got as the member of the $x$-family.
We have
$$\frac 12 \sum_{k=1}^{m-3}w_{k} = \frac 12 \sum_{k=1}^{m-3}{n\choose k}\overset{(\ref{eq005})}{\le} \frac 12{n\choose m-2} = \frac 12\prod_{j=0}^3\frac {m+2-j}{ 2m+j}{n\choose m+2}< \frac 1{16}{n\choose m+2}.$$
(We could have put $\frac 1{32}$ instead of $\frac 1{16}$, but it does not matter for the calculations.) Since each $(m+2)$-set in the $x$-family may get this charge from each of the two groups to which it does not belong, the total charge transferred in that way is at most $\frac 18{n\choose m+2} \overset{(\ref{eq05})}{=}
 w_{m+2}^l$. The lateral $(m+2)$-sets are not going to get any more charge. As for the central $(m+2)$-sets, we have to make sure that they will get not more than $w^c_{m+2}-w^l_{m+2} \overset{(\ref{eq05})}{=}
 \frac 14{n\choose m+2}$ additional charge.\\

We note that the $(m-2)$-sets will be discharged together with the corresponding $(m-1)$-sets.\\

\textbf{Stage~2. Transferring charge from pairs of $\mathbf{(m-1)}$-sets to $\mathbf{(m+3)}$-sets}.\\ Due to the fact that $\ff$ is an up-set,  from now on we may assume that the charge of any $(\le m-3)$-set is 0. Assume that, for some $x$ and  $i,j_1,j_2$, where $\{j_1,j_2,i\} =\{0,1,2\}$,  both $H_{j_1}^{(m-1)}(x)$ and $H_{j_2}^{(m-1)}(x)$ belong to $\mathcal F$. Then the set $H_i^{(m+3)}(x)$ is not in $\mathcal F$ and, consequently, has zero charge. Transfer the charge of the sets $H_{j_1}^{(k)}(x)$ and $H_{j_2}^{(k)}(x)$, $k=m-2,m-1$, to $H_i^{(m+3)}(x)$. The charge transferred is at most $2w_{m-2}+2w_{m-1}$, which is $$2{n\choose m-2}+2{n\choose m-1} \le \frac 32{n\choose m} = \frac {3(m+1)}{2(2m+1)}{n\choose m+1}\le {n\choose m+1}\le {n\choose m+3} \overset{(\ref{eq05})}{=}
 w_{m+3},$$
since $m\ge 3$. The $(m+3)$-sets are not going to get any more charge.\\

\textbf{Stage~3. Transferring charge from  $\mathbf{(m-1)}$-sets paired with $\mathbf{m}$-sets to central $\mathbf{(m+2)}$-sets}.\\ Assume that, for some $x\in[n]$ and   $i,j_1,j_2$, where $\{j_1,j_2,i\} =\{0,1,2\}$, both $H_{j_1}^{(m-1)}(x)$ and $H_{j_2}^{(m)}(x)$ belong to $\mathcal F$. Then the central $(m+2)$-set $H_i^{(m+2)}(x,j_1;x)$ is not in $\mathcal F$ and, consequently, received at most $w^l_{m+2}$ charge within the $x$-family (because of the charge possibly transferred on Stage~1). Transfer the charge of the sets $H_{j_1}^{(k)}(x)$, $k=m-2,m-1$, to $H_i^{(m+2)}(x,j_1;x)$. The charge transferred is at most $w_{m-2}+w_{m-1}$, which is
$${n\choose m-2}+{n\choose m-1} =\frac{3m+2}{2m+3}{n\choose m-1} = \frac{3m+2}{2m+3}\prod_{j=0}^2\frac{m+j}{2m-j+2}{n\choose m+2}\le \frac 14{n\choose m+2}.$$
The last inequality is valid for any $m\ge 1$ and is easy to verify by a direct calculation. The right hand side is exactly $w_{m+2}^c-w^l_{m+2}$, and so the total charge of the central $(m+2)$-sets is at most $w_{m+2}^c$ in each $x$-family. The $(m+2)$-sets are not going to get any more charge.\\

\textbf{Stage~4. Transferring charge from single $\mathbf{(m-1)}$-sets to $\mathbf{(m+1)}$-sets}.\\ After the above redistribution of charges, for each $x\in [n]$, there are no $(\le m-3)$-sets and at most one $(m-2)$-set and $(m-1)$-set with non-zero charges in the $x$-family,  moreover, we cannot have an $(m-1)$-set and two $m$-sets with non-zero charges in the $x$-family.

Assume that for some $x\in[n]$ and $i,j_1,j_2$, where $\{j_1,j_2,i\} =\{0,1,2\}$, the set $H_{i}^{(m-1)}(x)$ belongs to $\mathcal F$. Then one set from each of the two pairs
$\bigl(H_{j_1}^{(m+1)}(x;x),$ $ H_{j_2}^{(m+1)}(i;x)\bigr)$ and  $ \bigl(H_{j_1}^{(m+1)}(i;x), H_{j_2}^{(m+1)}(x;x)\bigr)$ is missing from $\mathcal F$. We transfer $\frac 12$ of the charge of the subsets $H_i^{(k)}(x)$, $k= m-2,m-1$, which is at most $\frac 12 (w_{m-2}+w_{m-1})$, to each of these missing sets. We have
\begin{equation}\label{eq77}\frac 12 (w_{m-2}+w_{m-1}) = \frac 12{n\choose m-2}+\frac 12{n\choose m-1} = \frac{(3m+2)m}{4(2m+3)(2m+1)}{n\choose m+1} \overset{(\ref{eq05})}{=}
 w_{m+1}^l.\end{equation}
Recall that $\alpha'> \alpha$, and, therefore, $w_{m+1}^c>w_{m+1}^l$, which means that no $(m+1)$-set gets more charge than its weight up to this stage.

\textbf{Stage~5. Transferring charge from pairs of $\mathbf{m}$-sets to central $\mathbf{(m+1)}$-sets}. \\ Denote the number of $m$-sets that have non-zero charge (that is, that are contained in $\mathcal F\cap \mathcal G(\sigma)$) by $q$. If $q\le m$, then we are clearly done since each $m$-set appears in exactly three $x$-families.

Assume that $q>m$. On the one hand, it makes an extra contribution $3(q-m){n\choose m}$ to the left hand side of (\ref{eq76}). On the other hand, the number of triples with two $m$-sets belonging to $\ff\cap \mathcal G(\sigma)$ is non-zero. Indeed, if, for $j\le 2$, we denote by $z_j$ the number of triples with $j$ $m$-sets in the family, then, since each $m$-set participates in three triples, we have $z_1+2z_2 = 3q$. Since $z_0+z_1+z_2=n$, we have $z_2\ge 3q-n$.  Assume that for some $x\in[n]$ and  $i,j_1,j_2$, where $\{j_1,j_2,i\} =\{0,1,2\}$, both $H_{j_1}^{(m)}(x)$ and $H_{j_2}^{(m)}(x)$ belong to $\mathcal F$. Then the central $(m+1)$-set $H_{i}^{(m+1)}(x;x)$ is not in the family. Moreover, no charge was transferred to it from the $x$-family, since we could not have had two $m$-sets and an $(m-1)$-set with non-zero charges at the same time after Stage~3. We transfer $\frac{3(q-m)}{z_2}{n\choose m}$ charge to $ H_{i}^{(m+1)}(x;x)$ from the $m$-sets.

First note that we have transferred $z_2\frac{3(q-m)}{z_2}{n\choose m}$ charge from the $m$-sets to the central $(m+1)$-sets, which results in $m$-sets having total charge of $3m{n\choose m}$. This is precisely what we needed to have, and we are only left to verify that we did not overcharge the central $(m+1)$-sets. Unfortunately, we can run into problems in this situation, so we have to consider two cases. First, assume that $q\ge m+2$. Then the charge on each central set in each $x$-family is at most
\begin{equation}\label{eq78} \frac{3(q-m)}{z_2}{n\choose m}\le \frac{3(q-m)}{3q-n}{n\choose m}\le \frac 65{n\choose m} = \frac {6(m+1)}{5(2m+1)}{n\choose m+1}.\end{equation}
The second inequality holds due to the fact that for $q\ge m+1$ the function $\frac{3(q-m)}{3q-n}$ decreases as $q$ grows. Therefore, if $\alpha'=(1-2\alpha)\ge \frac {6(m+1)}{5(2m+1)}$, then we are done.
Let us verify that this inequality is implied by (\ref{eq04}). Adding $2\alpha$ to the right hand side of the inequality, we get
$$\frac {6(m+1)}{5(2m+1)}+\frac{(3m+2)m}{2(2m+3)(2m+1)} = \frac{12(m+1)(2m+3)+5(3m+2)m}{10(2m+3)(2m+1)} = $$$$=\frac{39m^2+70m+36}{40m^2+80m+30}<1,$$
where the last inequality holds for any $m\ge 1$. Thus, we fulfilled all the requirements on the charging scheme and we are done in the case $q\ge m+2$.

In the case $q=m+1$, however, we run into trouble: the inequality $\alpha'\ge \frac{3(q-m)(m+1)}{z_2(2m+1)} = \frac{3(m+1)}{z_2(2m+1)}$ may not hold. Recall that $z_2\ge 3q-n= 2$. We are still fine if $z_2\ge 3$ since the calculations in \eqref{eq78} still go through in that case.  Thus we only need to examine the case $z_2=2$, which we assume until the end of this section. The equation $z_2=2$ means that there are exactly two triples with two $m$-sets from $\mathcal F$. This is the only part of the proof when we are not going to compare the amount of charge passed to the $(m+1)$-sets  to the portion of its weight inside the $x$-family. Instead, we compare the charge to the full weight of the $(m+1)$-set.

We have two possible configurations with $z_2=2$. One possibility is that we have two $m$-sets from $\ff$ forming an interval of length $2m$ on the circle, and then the two triples contributing to $z_2$ share the same two $m$-sets. In this case, the central $(m+1)$-set that we forbid is the same in both triples, and it is of type 1a (see Fig.~\ref{fig1}). Recall that this set has weight $2(\alpha+\alpha'){n\choose m+1}$. The other possibility is that we have two pairs of $m$-sets, with each pair separated on the two sides by a third $m$-set forming a triple with the pair, and by the element missing from the triple, respectively. In this case, in each of the corresponding two $x$-families we forbid a central $(m+1)$-set of type 1b. Each of these two sets (that are clearly different) has weight $(\alpha+\alpha'){n\choose m+1}$. In either case, we need to transfer $\frac{3(m+1)}{2m+1}{n\choose m+1}$ amount of charge from the $m$-sets to some of the $(m+1)$-sets. We transfer this weight to the central $(m+1)$-set(s), possibly overcharging it.

Assume first that we do not have any $(m-1)$-element sets in $\ff\cap \mathcal G(\sigma)$. Then, in either of the possibilities described above, the central $(m+1)$-sets have zero charge before Stage~5. Therefore, we are good if the weight of these (one or two) $(m+1)$-sets is greater than the amount of charge we transfer to them from the pairs of $m$-sets. Namely, we are good if
\begin{equation}\label{eq11}2(\alpha+\alpha') = 2-2\alpha\ge \frac{3(m+1)}{2m+1}\ \  \ \Leftrightarrow\ \ \ 2\ge \frac{3(m+1)}{2m+1}+\frac{(3m+2)m}{2(2m+3)(2m+1)}.\end{equation}
We have
$$\frac{3(m+1)}{2m+1}+\frac{(3m+2)m}{2(2m+3)(2m+1)} = \frac{15m^2+32m+18}{8m^2+16m+6}<2,$$
where the last inequality holds for $m\ge 3$. Thus, this case is covered.

\begin{figure}[t] \begin{center}  \includegraphics[width=100mm]{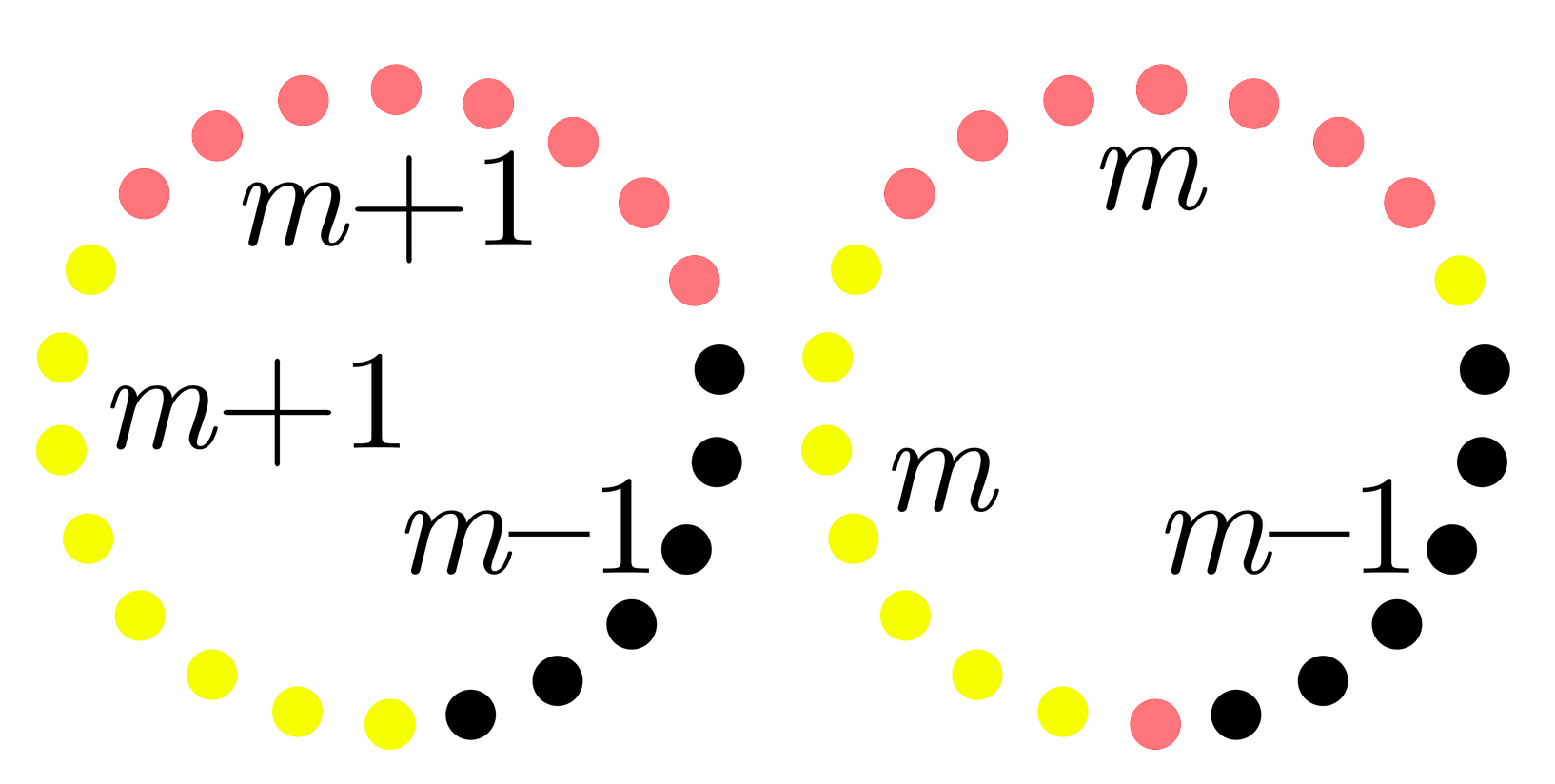} \caption{In each of the two cases the $(m-1)$-set is represented by black points, and two $(m+1)$-sets, which form a matching with the $(m-1)$-set, are represented by red (grey) and yellow (light grey) points.}\label{fig2}\end{center} \end{figure}

Finally, assume that there is at least one $(m-1)$-element set $F\in \ff\cap \mathcal G(\sigma)$. Then, as one can see from Fig.~\ref{fig2}, it forbids at least one $(m+1)$-set of each of the types 1a and 1b to appear. Denote them by $M_1,M_2$. We have seen two paragraphs above that in either case of the arrangement of the $m$-sets we forbid sets of the same type (either one of type 1a, or two of type 1b). Therefore, at least one of $M_1,M_2$ that did not get any charge at Stage~5. We assume that it is a set $M_1$ of type 1b (the other case is easier and is treated similarly). The set $M_1$ appears in two $x$-families and got some
charge only at Stage~4. Moreover, $M_1$ got at most $\alpha{n\choose m+1}$ charge from each of the two $x$-families. Thus, the charge of $M_1$ after all five stages is at most $2\alpha{n\choose m+1}$. This, in turn, means that it has extra capacity of at least $(\alpha'-\alpha){n\choose m+1}$. We redistribute some part of the charge from the two $(m+1)$-sets (that appeared in the $x$-families with the pairs of $m$-sets from $\ff$) to $M_1$. In order to be able to fulfil the requirements on the charging scheme, we need the total capacity of these $(m+1)$-sets to be greater than the charge we transfer. More precisely, it is sufficient if the following inequality holds:
$$(3\alpha'-\alpha)\ge \frac{3(m+1)}{2m+1} \ \ \ \Leftrightarrow \ \ \ 3-7\alpha\ge \frac{3(m+1)}{2m+1}.$$
Note that we replaced the capacity  of the (one or two) missing central $(m+1)$-set(s) by $2\alpha'$  since the $\alpha$-part of the charge may have been already used up by the $(m-1)$-sets at Stage~4. We have verified above (see (\ref{eq11})) that the same inequality holds if one replaces $3-7\alpha$ with $2-2\alpha$. One can easily see (cf. (\ref{eq04})) that $\alpha\le \frac 3 {16}$, so we have $3-7\alpha>2-2\alpha$. The case $q=m+1$ is examined in its entirety, and the proof of the theorem in the case $s=3, m\ge 3$ is complete.

\subsection{The case $m\le 2, s=3$}
In the argument above, we assumed that $m\ge 3$. However, we want to prove the theorem for $m\ge 1$, which leaves us with two cases: $m=1$ and $m=2$. If $m=1$, then we have $n=4$, and we have to show that at least four sets, including the empty set, are missing from a family $\ff\subset 2^{[4]}$ with $\nu(\ff)\le 2$. If there is at most one singleton in $\ff$, then we are done. If there are at least two singletons, say, $\{1\}$ and $\{2\}$, then $\ff\cap 2^{\{3,4\}}$ is empty, which gives $4$ missing sets. The case $m=1$ is covered.\\

If $m=2$ then $n=7$, and we have to show that at least $1+{7\choose 1}+{6\choose 2}=23$ sets are missing from $\ff$. If there is at least one singleton in $\ff$, say $\{1\}$, then $\ff\cap 2^{[2,7]}$ is intersecting, and so, by the Erd\H os-Ko-Rado theorem, a half of the sets are missing from it. This gives $32$ missing sets.

Thus, we may assume that there are no sets of size smaller than $2$ in $\ff$. Now we may slightly modify the proof for the case $m\ge 3$ so that it works for $m\ge 2$. Namely, among the $(m+1)$-sets, we give weights only to the central $(m+1)$-sets (the weights on other layers stay the same). Claim~\ref{cla2} stays true in this case. Since we do not have sets of size smaller than $m$, we can go to Stage~5 of the analysis, where we want to show that with the new weights \eqref{eq78} holds for any $q\ge m+1$ for $m=2, n=7$:
$$\frac{3(q-m)}{3q-n}{n\choose m}\le \frac 32{n\choose m} \le  {n\choose m+1}\ \ \ \ \Leftrightarrow \ \ \ \ \ \ \frac 32{7\choose 2}=\frac{63}2\le {7\choose 3} = 35.$$
The last inequality obviously holds. Thus, for $m=2$ we may terminate the proof right after~\eqref{eq78}. The proof is complete.

\section{Proof of Theorem \ref{thm1} for $s = 4$}
We first prove the theorem for $m\ge 3$.
We put $n := 4m +2$ for some $m\ge 3$ throughout this section. The logic of this proof is very similar to that of the proof in the case $n = 3m+1$, and the proof is in a sense even simpler. We present it somewhat more concisely.

We fix an arbitrary permutation $\sigma$ of the ground set. For simplicity, we assume that $\sigma$ is the identity permutation.
 Quite predictably, define four groups of sets, indexed by $i=0,1,2,3$ and forming an \textit{$x$-family}. The four $m$-sets $H_i^{(m)}(x)$ in an $x$-family are disjoint and form an interval of length $4m$, leaving two contiguous elements $x-1, x$ out (thus, the $x$-family is indexed by the last of the two missing elements in the clockwise order). In what follows, we define the $i$-th group.  The sets in the $i$-th group of size $m-j,\ j=1,\ldots, m$, form a full chain together with $H_i^{(m)}(x)$:
$$H_i^{(m-j)}(x):=\{x+1+j+im,\ldots, x+(i+1)m\}.$$
   We again have both central and lateral $(m+1)$- and $(m+2)$-sets. The $(m+1)$-sets
$$H_i^{(m+1)}(x';x):=H_i^{(m)}(x)\cup\{x'\} \ \ \ \ \text{for }x' = x,x-1$$
in group $i$ are called \textit{central}. Note that the extra element in both sets is left out by the $m$-sets, and so $H_i^{(m+1)}(x';x)$ for both $x'=x-1,x$ is disjoint of the $m$-set from the $j$-th group, $j\ne i$. The three others
$$H_i^{(m+1)}(j;x):=H_i^{(m)}(x)\cup\{jm+x+1\}\ \ \ \ \text{for }j\in \{0,\ldots,3\}-\{i\}$$
 are called \textit{lateral} and are disjoint of the corresponding $H_j^{(m-1)}(x)$ and of the $m$-set in the $j'$-group, $j'\ne i,j$.

For each $j\in \{0,\ldots,3\}-\{i\}$, we define two lateral $(m+2)$-element sets: $$H_i^{(m+2)}(x',j;x):=H_i^{(m+1)}(j;x)\cup\{x'\}\ \ \ \ \text{for }x'=x,x-1,$$
and for each $i$ we define one central set:
$$H_i^{(m+2)}(x-1,x;x):=H_i^{(m)}(x)\cup\{x-1,x\}.$$
The former ones are disjoint of the $(m-1)$-set from group $j$ and the two $m$-sets from the remaining groups, while the latter one is disjoint of the three $m$-sets from the groups $j, j\ne i$. Finally, we have one $(m+5)$-element set for each $i:$
$$H_i^{(m+5)}(x):=H_i^{(m)}(x)\cup\{x-1,x,x+1,m+x+1,2m+x+1,3m+x+1\}.$$

Each set in each group gets a weight.  We denote by $w_{k}$ the weight of the $k$-element sets, with possible superscripts $l,c$ depending on whether the set is lateral or central, respectively. The weights are as follows ($j\ge 0$):
\begin{align}\label{eq07}w_{m-j}&:={n\choose m-j};  \ \ \ \ w^l_{m+1}:=\frac {m}{5(3m+2)}{n\choose m+1};\ \ \ \  w^c_{m+1}:=\frac 12{n\choose m+1}-\frac 32 w_{m+1}^l;\notag\\
w_{m+5} &:= {n\choose m+5}; \ \ \ \ w^l_{m+2}:=\frac 1{22}{n\choose m+2};\ \ \ \ \ \ \ \ \ \ \ \ \ \  w^c_{m+2}:=\frac{8}{11}{n\choose m+2}.
\end{align}

It is easy to check that, for each $i$ and $k\in\{1,\ldots,m+2\}\cup \{m+5\},$ in each group the weight of $k$-element sets sums up to ${n\choose k}$ for fixed $x$, and that $w_{m+1}^c$ is positive.

As before, each set $F$ that appears in some $x$-families accumulates all the weight $w(F)$ that it was assigned. We denote by $\mathcal G(\sigma)$ the family of all sets that got  non-zero weight for a given permutation $\sigma$.
Analogously to Claim~\ref{cla2}, to prove the theorem in this case, it is sufficient to show that for any $\sigma$ we have
\begin{equation}\label{eq08}\sum_{F\in \mathcal F\cap \mathcal G(\sigma)}w(F)\le 4n\Bigl({n-1\choose m-1}+\sum_{k\in\{1,2,5\}}{n\choose m+k}\Bigr) = 4m{n\choose m}+4n\sum_{k\in\{1,2,5\}}{n\choose m+k}.\end{equation}

For a set $F\in \mathcal G(\sigma)$, we define the charge $c(F)$ to be equal to $w(F)$ if $F\in\mathcal F$, and $c(F):=0$ otherwise. Clearly, $\sum_{F\in \mathcal G(\sigma)}c(F) = \sum_{F\in \mathcal F\cap \mathcal G(\sigma)}w(F)$. We again design a scheme for the transfer of (a part of) the charge of the $(\le m)$-sets to the $(\ge m+1)$-sets that have zero charge. We show that the charge transferred to each $(\ge m+1)$-set is at most its weight. As a result of this procedure, the $(\le m-1)$-sets will have zero total charge, the $m$-sets will have total charge $4m{n\choose m}$, and each $(\ge m+1)$-set will have charge not greater than its own weight. This will obviously conclude the proof of the theorem.

Next we design a charging scheme that satisfies the above requirements. For $n=4m+2$ it is sufficient in all cases to redistribute the charge within  each $x$-family, assuring that the charge that we transferred to the larger set in one $x$-family is smaller than the weight that this bigger set got from this $x$-family.

\textbf{Stage~1. Transferring charge from triples of $\mathbf{(m-1)}$-sets to $\mathbf{(m+5)}$-sets}.\\ Assume that, for some $x\in [n]$ and  $i_1,i_2,i_3,j$, where $\{i_1,i_2,i_3,j\} =\{0,1,2,3\}$, the sets $H_{i_u}^{(m-1)}(x)$ for $u\in[3]$, belong to $\mathcal F$. Then $H_{j}^{(m+5)}(x)$ is missing from $\ff$, and, consequently, has zero charge. Transfer all  the charge of the sets $H_{i_u}^{(k)}(x)$, $k\le m-1$ to the missing $(m+5)$-set.

The charge transferred  is at most $3\sum_{k=1}^{m-1}w_{k}$, which is
$$3{n\choose m-1}+3\sum_{k=1}^{m-2}{n\choose k}\overset{(\ref{eq005})}{\le} \frac 92{n\choose m-1} = \frac {9\prod_{p=0}^5 (m+p)}{2\prod_{p=-2}^3(3m+p)}{n\choose m+5}< {n\choose m+5}\overset{(\ref{eq07})}{=} w_{m+5},$$
where the last inequality holds for any $m\ge 2$. Note that we apply \eqref{eq005} for $k=m-1$ in the first inequality above (and in several places below). The $(m+5)$-sets are not going to get any more charge.\\

\textbf{Stage~2. Transferring charge from pairs of $\mathbf{(m-1)}$-sets to lateral $\mathbf{(m+2)}$-sets}.\\ Assume that, for some $x\in[n]$ and $i_1,i_2,j_1,j_2$, where $\{i_1,i_2,j_1,j_2\} =\{0,1,2,3\}$,  both $H_{j_1}^{(m-1)}(x)$ and $H_{j_2}^{(m-1)}(x)$ belong to $\mathcal F$. Then in each of the four pairs $\bigl(H_{i_1}^{(m+2)}(x',j';x),$ $H_{i_2}^{(m+2)}(x'',j'';x)\bigr)$, where $\{j',j''\} =\{j_1,j_2\},\{x',x''\} =\{x-1,x\}$ one of the $(m+2)$-sets is missing from $\ff$, and, consequently, has zero charge. Note that all these $(m+2)$-sets are lateral.
Transfer one quarter of the charge of the sets $H_{j_1}^{(k)}(x)$ and $H_{j_2}^{(k)}(x)$, $k\le m-1$, to each of these missing sets.

The charge transferred to each lateral $(m+2)$-set is at most $\frac 12\sum_{k=1}^{m-1}w_{k}$, which is $$\frac 12 \sum_{k=1}^{m-1}{n\choose k} \overset{(\ref{eq005})}{\le} \frac 34{n\choose m-1} = \frac {3m(m+1)(m+2)}{4(3m+3)(3m+2)(3m+1)}{n\choose m+2}< \frac 1{22}{n\choose m+2} \overset{(\ref{eq07})}{=} w_{m+2}^l,$$
where the last inequality holds for any $m\ge 1$. 
The lateral $(m+2)$-sets are not going to get any more charge.\\

\textbf{Stage~3. Transferring charge from single $\mathbf{(m-1)}$-sets to $\mathbf{(m+1)}$-sets}.\\ After the above redistribution of charges, we have at most one $(m-1)$-set with non-zero charge in each $x$-family.

Assume that for some $x\in[n]$ and  $i,j_1,j_2,j_3$, where $\{j_1,j_2,j_3,i\} =\{0,1,2,3\}$, the set $H_{i}^{(m-1)}(x)$ belongs to $\mathcal F$ and still has non-zero charge. Then one set from each of the six triples
$\bigl(H_{j_1}^{(m+1)}(\pi(i);x),$ $ H_{j_2}^{(m+1)}(\pi(x-1);x), H_{j_3}^{(m+1)}(\pi(x);x)\bigr)$, where $\pi$ is a permutation of the set $\{i,x-1,x\}$, is missing from $\mathcal F$. It is not difficult to see that it means that at least three out of the listed sets are missing from $\ff$. Note that among the possible missing sets there are both central and lateral $(m+1)$-sets.

We transfer $\frac 13$ of the charge of $H_i^{(k)}(x)$, $k \le m-1$, to each of the three missing sets. This is at most $\frac 13 \sum_{k=1}^{m-1}w_{k},$  which is

\begin{equation*}\label{eq77}\frac 13 \sum_{k=1}^{m-1}{n\choose k}\overset{(\ref{eq005})}{\le}
\frac 12{n\choose m-1} =\frac{m(m+1)}{2(3m+3)(3m+2)}{n\choose m+1}=\frac {m}{6(3m+2)}{n\choose m+1}\overset{(\ref{eq07})}{<}w^l_{m+1},\end{equation*}
We are not going to transfer any more weight to the lateral $(m+1)$-sets.\\

\textbf{Stage~4. Transferring charge from pairs and triples of $\mathbf{m}$-sets}.\\ At this stage only the sets of size greater than or equal to $m$ have non-negative charge. Denote the number of $m$-sets that have non-zero charge (that is, that are contained in $\mathcal F\cap \mathcal G(\sigma)$) by $q$. If $q\le m$, then we are clearly done.

Assume that $q>m$. On the one hand, it makes an extra contribution $4(q-m){n\choose m}$ to the left hand side of (\ref{eq08}). On the other hand, the number of quadruples with two or three $m$-sets belonging to $\ff\cap \mathcal G(\sigma)$ is non-zero. Indeed, if we denote by $z_j$ the number of quadruples with $j$ $m$-sets in the family, for $j\le 3$, then we have $z_1+2z_2+3z_3= 4q$. Since $z_0+z_1+z_2+z_3=n$, we have \begin{equation}\label{eq09}
z_2+2z_3\ge 4q-n.\end{equation} We proceed as follows.\\

\textbf{(i) Triples of $\mathbf{m}$-sets.}
Assume that, for some $x\in [n]$ and $j_1,j_2,j_3,i$, where $\{j_1,j_2,j_3,i\} =\{0,1,2,3\}$, the sets $H_{j_u}^{(m)}(x)$ belong to $\ff$ for all $u = 1,2,3$. Then the central $(m+2)$-set $H_{i}^{(m+2)}(x-1,x;x)$ is not in the family $\ff$. Moreover, it has zero charge. We transfer $\frac{8(q-m)}{4q-n}{n\choose m}$ charge to this set. We have $\frac{8(q-m)}{4q-n}\le 4$ for $q\ge m+1$,  since this function for $q\ge m+1$ decreases as $q$ grows. Therefore, we have
\begin{equation}\label{eq0667} \frac{8(q-m)}{4q-n}{n\choose m}\le 4{n\choose m} = 4\frac{(m+1)(m+2)}{(3m+2)(3m+1)}{n\choose m+2}\le \frac 8{11}{n\choose m+2} \overset{(\ref{eq07})}= w_{m+2}^c. \end{equation}
The last inequality holds for $m\ge 3$.\\

\textbf{(ii) Pairs of $\mathbf{m}$-sets.} Assume that for some $x\in[n]$ and $i_1,i_2,j_1,j_2$, where $\{j_1,j_2,i_1,i_2\} =\{0,1,2,3\}$, exactly two $m$-sets $H_{j_1}^{(m)}(x)$ and $H_{j_2}^{(m)}(x)$ from the $x$-family belong to $\mathcal F$. Then in each of the two pairs of central $(m+1)$-sets $\bigl(H_{i_1}^{(m+1)}(x';x),H_{i_2}^{(m+1)}(x'';x)\bigr)$ for $\{x',x''\} = \{x-1,x\}$ one of the sets is not in the family. Moreover, each of them has received at most $w^{l}_{m+1}$ charge within this $x$-family (they could have received charge only in Stage~3). We transfer $\frac{2(q-m)}{4q-n}{n\choose m}$ charge to each of the two missing central sets. We have to verify that the charge transferred is at most $w_{m+1}^c-w_{m+1}^l$. Note that $\frac{2(q-m)}{4q-n}\le 1$. Therefore, it is enough to verify
\begin{equation}\label{eq0666} {n\choose m}\le w_{m+1}^c-w_{m+1}^l\ \ \ \overset{(\ref{eq07})}\Leftrightarrow \ \ \ 2{n\choose m} = \frac{2m+2}{3m+2}{n\choose m+1}\le {n\choose m+1}-5w_{m+1}^l. \end{equation}
The last inequality holds (with equality) since by (\ref{eq07}) we have $5w_{m+1}^l = \frac m{3m+2}{n\choose m+1}$.\\

Now we only have to make sure that we have transferred enough charge. Indeed, we have transferred a total amount of charge equal to
$$\frac{4(q-m)}{4q-n}{n\choose m}z_2+\frac{8(q-m)}{4q-n}{n\choose m}z_3= \frac{4(q-m)}{4q-n}{n\choose m}(z_2+2z_3) \overset{(\ref{eq09})}\ge 4(q-m){n\choose m}.$$
Therefore, the total amount of charge that is left on the $m$-sets is at most $4m{n\choose m}$, moreover, all sets of size not greater than $m-1$ have zero charge, and none of the sets has the charge that is greater than its weight. The inequality (\ref{eq08}) is verified, and the proof of Theorem \ref{thm1} in the case $s=4, m\ge 3$ is complete.

\subsection{The case $m\le 2, s=4$}
In the argument above, we assumed that $m\ge 3$. However, we wang to prove the theorem for $m\ge 1$, which leaves us with two cases: $m=1$ and $m=2$. If $m=1$, then we have $n=6$, and we have to show that at least $6$ sets, including the empty set, are missing from a family $\ff\subset 2^{[6]}$ with $\nu(\ff)\le 3$. If there is at most one singleton in $\ff$, then we are done. If there are at least two singletons, say, $\{1\}$ and $\{2\}$, then $\ff\cap [3,6]$ is intersecting, and, consequently, at least $8$ sets are missing from $\ff$ among the sets from $2^{[3,6]}$. The case $m=1$ is covered.\\

If $m=2$ then $n=10$ and we have to show that at least $1+{10\choose 1}+{9\choose 2}=47$ sets are missing from $\ff$. If there is at least one singleton in $\ff$, say $\{1\}$, then, applying \eqref{eq002} to $\ff\cap 2^{[2,10]}$, we get that at least $1+{9\choose 1}+{9\choose 2}+\frac 13{9\choose 3}$ sets are missing from $\ff$, which is more than 47.

Thus, we may assume that there are no sets of size smaller than $2$ in $\ff$. Now we may slightly modify the proof for the case $m\ge 3$ so that it works for $m\ge 2$. Namely, among the $(m+1)$- and $(m+2)$-sets we give weights only to the central $(m+1)$- and $(m+2)$-sets (each central $(m+1)$-set receives a weight of $\frac 12{n\choose m+1}$, each central $(m+2)$-set receives a weight of ${n\choose m+2}$, and the weights on other layers stay the same). Since $\ff$ contains no sets of size smaller than $m$, we may go to part 4 of the analysis, where we have to verify the following analogues of \eqref{eq0666} and \eqref{eq0667} for $m=2,\ n=10$:
$$\frac{2m+2}{3m+2}{n\choose m+1}\le {n\choose m+1},\ \ \ \ \ \ \ \ \ \ \ \ \ 4\frac{(m+1)(m+2)}{(3m+2)(3m+1)}{n\choose m+2}\le {n\choose m+2}.$$
Both hold for $m=2$. The rest of the proof stays the same. The proof is complete.\\

{\large\textsc{Acknowledgements.\ }} We thank  the anonymous referees for their helpful comments on the presentation of the paper.

\end{document}